\documentclass[bezier,amstex, oneside,reqno]{amsart}
\usepackage{amsmath, amssymb, amsthm, verbatim, euscript, amscd, textcomp}
\usepackage[dvips]{graphicx}
\usepackage{epsfig}
 \usepackage[all,cmtip]{xy}
\usepackage{color}
 \usepackage[fleqn]{cases}

\def\ie{\emph{i.e., }}
\def\eg{\emph{e.g., }}

\def\R{\mathbb R}
\def\Z{\mathbb Z}

\def\C{\mathbb C}
\def\T{\mathbb T}

\def\e{\varepsilon}

\def\TT{\mathcal T}
\def\Diff{\textup{Diff}\,}

\newtheorem*{problem}{Problem}
\newtheorem*{maintheorem}{Main Theorem}
\newtheorem{theorem}{Theorem}[section]

\newtheorem{add}[theorem]{Addendum}

\newtheorem{cor}[theorem]{Corollary}
\newtheorem{lemma}[theorem]{Lemma}
 \theoremstyle{remark}
\newtheorem{remark}[theorem]{Remark}

\theoremstyle{remark}

\begin{document}
\author{Andrey Gogolev$^\ast$}
\title[Bootstrap for Anosov automorphisms]{Bootstrap for local rigidity of Anosov automorphisms on the 3-torus}
\thanks{$^\ast$The author was partially supported by NSF grant DMS-1266282.} 
\begin{abstract}
We establish a strong form of local rigidity for hyperbolic automorphisms of the 3-torus with real spectrum. Namely, let $L\colon\T^3\to\T^3$ be a hyperbolic automorphism of the 3-torus with real spectrum and let $f$ be a $C^1$ small perturbation of $L$. Then $f$ is smoothly ($C^\infty$) conjugate to $L$ if and only if obstructions to $C^1$ conjugacy given by the eigenvalues at periodic points of $f$ vanish. By combining our result and a local rigidity result of Kalinin and Sadovskaya for conformal automorphisms~\cite{KS3} this completes the local rigidity program for hyperbolic automorphisms in dimension 3. Our work extends de la Llave-Marco-Moriy\'on 2-dimensional local rigidity theory~\cite{MM, L0, L1}.
\end{abstract}
\date{}
 \maketitle

\section{Introduction}

An automorphism of $\R^d$ induced by a matrix $L\in GL(d,\Z)$ descends to an automorphism of the torus $\T^d=\R^d/\Z^d$, which we still denote by $L\colon \T^d\to \T^d$. An automorphism $L\colon \T^d\to\T^d$, $d\ge 2$, is called {\it hyperbolic} or {\it Anosov} if the eigenvalues of corresponding matrix $L\in GL(d,\Z)$ lie off the unit circle in $\C$. By Anosov's structural stability theorem~\cite{An} any sufficiently $C^1$-small perturbation of a hyperbolic automorphism $L$ is topologically conjugate to $L$. Some obvious obstructions for $C^1$ conjugacy (and a fortiori) for higher regularity conjugacy are carried by periodic points. 
Namely, given a hyperbolic automorphism $L\colon \T^d\to\T^d$ and sufficiently $C^1$-small smooth perturbation $f\colon\T^d\to\T^d$ we say that {\it periodic data obstructions vanish} for $f$ if for each $f$-periodic point $p$ the Jordan normal form of the differential $Df^n(p)$ is the same as the Jordan normal form of $L^n$, where $n$ is the period of $p$. A hyperbolic automorphism $L\colon\T^d\to \T^d$ is called {\it locally rigid} if any sufficiently $C^1$-small smooth perturbation of $L$, for which periodic data obstructions vanish, is conjugate to $L$ via a smooth diffeomorphism. (Throughout the paper by {\it ``smooth"} we mean ``$C^\infty$ differentiable.")

\begin{maintheorem}
\label{thm_main}
 Assume that  $L\colon\T^3\to\T^3$ is
 a hyperbolic automorphism with real spectrum.
Then $L$ is locally rigid.
\end{maintheorem}
To the best of our knowledge this is the first local rigidity result which yields smooth conjugacy in non-conformal setting. In fact, $C^{1+\textup{H\"older}}$ regularity of the conjugacy was established in~\cite{GG} and our contribution is the bootstrap of the regularity to $C^\infty$.

Combining then Main Theorem and the local rigidity result for automorphisms with complex eigenvalues by Kalinin and Sadovskaya~\cite{KS3} yields the following corollary.
\begin{cor}
 All hyperbolic automorphisms of the 3-torus are locally rigid.
\end{cor}
\begin{add}[Main Theorem in finite smoothness]
 \label{add}
 Let $L\colon\T^3\to\T^3$ be a hyperbolic automorphism with real spectrum and let $f\colon\T^3\to\T^3$ be a $C^r$, $r>1$, diffeomorphism which is $C^1$ close to $L$. Then for any $r>3$ diffeomorphism $f$ is conjugate to $L$ via a $C^{r-3-\e}$ diffeomorphism. Further, there exists a critical regularity $\kappa=\kappa(L)\in\Z$ such that if periodic data obstruction for $f$ vanish and $r\notin(\kappa,\kappa+3)$ then $f$ is conjugate to $L$ via a $C^{r-\e}$ diffeomorphism, where $\e>0$ is arbitrarily small. \end{add}
\begin{remark} The loss of 3 derivatives is due to our use of regularity results for cohomological equations over Diophantine torus translations. However outside the ``critical interval" earlier results~\cite{GG} (for $r\le\kappa$) and~\cite{L1} (for $r\ge \kappa+3$) recover missing derivatives.
\end{remark}
\begin{add}[Analytic local rigidity]
\label{add_analytic}
If the perturbation $f$ is analytic then the smooth conjugacy given by the Main Theorem is also analytic.
\end{add}

Deformation rigidity and then local (and global) rigidity was first discovered by de la Llave, Marco and Moriy\'on~\cite{MM, L0} for hyperbolic automorphisms of the 2-torus and, since then, was generalized to certain classes of hyperbolic conformal automorphisms, \ie automorphisms whose spectrum is confined to a circle of radius greater than 1 and a circle of radius less than 1 in $\C$; see~\cite{L3, KS3} and references therein.  Weaker form of local rigidity, which only yields $C^{1+\textup{H\"older}}$ regularity of the conjugating homeomorphism, was established for a rather large class of hyperbolic automorphisms of higher dimensional tori in~\cite{GKS}. On the other hand, de la Llave discovered that for reducible hyperbolic automorphisms (\ie automorphism induced by matrices with reducible characteristic polynomial) local rigidity may fail~\cite{L1}. Finally, we refer to~\cite{L1,G} for a discussion of the more general problem of smooth conjugacy for Anosov diffeomorphisms. Here we restrain ourselves to pointing out that smooth local (and global) rigidity of non-linear Anosov diffeomorphisms is still an open problem. The following is the most basic version of this problem.
\begin{problem}
 Let $L\colon\T^3\to\T^3$ be a hyperbolic automorphism. Show that there exists a $C^1$ neighborhood $\mathcal U$ of $L$ in the space of Anosov diffeomorphisms of $\T^3$ such that any two diffeomorphisms $f,g\in\mathcal U$ with the same periodic data are $C^\infty$ conjugate.
\end{problem}

One motivation for studying local rigidity of Anosov dynamical systems comes from spectral rigidity program in geometry of negatively curved manifolds (which was begun in~\cite{GK}; also see~\cite{Ham, CrSh} for more recent developments). Arguably, the local rigidity problem for Anosov systems is the analogue of marked length spectrum rigidity problem in geometry of negatively curved manifolds. We refer to~\cite[Section 4.1]{FGO} for more discussion of this analogy.

In Section~\ref{sec_2} we summarize some well known tools which we use in the proof of the Main Theorem. Section~\ref{sec_3} is mostly devoted to the proof of the Main Theorem.  At the end of Section~\ref{sec_3} we briefly explain how the proof of the Main Theorem in combination with a result of de la Llave yields Addendum~\ref{add}. Section~4 is devoted to the proof of Addendum~\ref{add_analytic}.

{\bfseries Acknowledgements.} This project started in the Spring of 2011 when the author realized that regularity results for the cohomological equation over Diophantine translations can be used to bootstrap regularity of the leaves of the intermediate foliation. The author would like to thank Rafael de la Llave for invaluable discussions during that period. The author would like to thank Anatole Katok for his attention to this project and for encouraging the author to finish it. The project was completed in Spring 2014 with the idea of dimension reduction via factor dynamics on the universal cover, which allowed the author to dispose of the spectral gap assumption needed for the earlier version of the result. In the Spring 2014 the author was visiting IMS at Stony Brook University and he would like to thank IMS for providing excellent working conditions and the people at Stony Brook for their hospitality.

The author is thankful to the referees whose careful reading lead to a significant improvement of the paper.

\section{Preliminaries}
\label{sec_2}
Let $W$ be a foliation on a smooth Riemannian manifold $M$. We will denote by $W(x)$ the leaf of $W$ passing through $x$ and by $W(x,R)$ for a ball of (intrinsic) radius $R$ in $W(x)$ centered at $x$. We will also use superscript $loc$ to denote the local leaf $W^{loc}(x)$; \ie $W^{loc}(x)=W(x,\delta)$, where $\delta>0$ is an appropriately small constant.

\subsection{Anosov diffeomorphisms}

Let $M$ be a compact Riemannian manifold. Hyperbolic toral automorphisms considered in the introduction are particular instances of Anosov diffeomorphisms. Recall that a diffeomorphism $g\colon M\to M$ is called {\it Anosov} if there exist a decomposition of the tangent bundle $TM$ into two
$g$-invariant continuous distributions $E^{s}_g$ and $E^{u}_g$, and
constants $C>0$, $\lambda\in(0,1)$, such that for all $n>0$,
 $$
   \begin{aligned}
  \| Dg^n(v) \| \leq C \lambda^n \| v \|&
      \;\text{ for all }v \in E^{s}_g, \\
  \| Dg^{-n}(v) \| \leq C \lambda^n \| v \|&
      \;\text{ for all }v \in E^{u}_g.
  \end{aligned}
 $$
The distributions $E^{s}_g$ and $E^{u}_g$ are called {\it the stable and unstable distributions} of $g$. These distributions integrate to {\it the stable and unstable foliations} $W^s_g$ and $W^u_g$, respectively. Foliations $W^s_g$ and $W^u_g$ are continuous foliations with smooth leaves (to be defined shortly in the next subsection).  The leaves $W^s_g(x)$ and $W^u_g(x)$, $x\in M$, are immersed copies of Euclidian spaces of dimension $\dim E^s_g$ and $\dim E^u_g$, respectively. These leaves can be characterized as follows
\begin{align}\label{eq_leaves}
\begin{split}
W^s_g(x)=\{ y\in M: \; d(f^n(y),f^n(x))\to 0, n\to+\infty \}\\
W^u_g(x)=\{ y\in M: \; d(f^n(y),f^n(x))\to 0, n\to-\infty \}
\end{split}
\end{align}
where $d$ is the induced metric on $M$.

Any diffeomorphism $f$ which is sufficiently $C^1$-close to $g$ is also Anosov and the celebrated structural stability theorem asserts that $f$ is conjugate to $g$
$$
h\circ g= f\circ h,
$$
where $h\colon M\to M$ is a homeomorphism which is $C^0$ close to the identity map.
Using~(\ref{eq_leaves}) we obtain that $h$ preserves the stable and unstable foliations.

Note that in the case when $g$ is a hyperbolic toral automorphism the stable and unstable foliations are linear foliations by totally irrational linear subspaces.

\subsection{Journ\'e's Lemma}
\label{sec_journe}
A foliation $W$ on a manifold $M$ is called a {\it uniformly continuous foliation with $C^r$ leaves} if the leaves of $W$ are uniformly $C^r$ injectively immersed submanifolds of $M$ and the tangent bundle $TW$ is a (uniformly) continuous subbundle of the full tangent bundle $TM$.


We say that a function $\varphi \colon M\to \R$ {\it is uniformly} $C^q$, $q\le r$, {\it along } $W$ and we write $\varphi\in C^q_W(M)$ if the restrictions of $\varphi$ to the leaves of $W$ have uniformly bounded derivatives of all orders up to $q$.
\begin{lemma}[Journ\'e~\cite{J}]
\label{lemma_journe}
Let $W$ and $V$ be two mutually transverse uniformly continuous foliations with $C^r$ leaves on a manifold $M$. Let $\varphi \colon M\to \R$ be a function. Assume that $\varphi\in C^r_W(M)\cap C^r_V(M)$. Then $\varphi$ is $C^{r-\e}$ for any $\e>0$.
\end{lemma}

If $W_1$ is a uniformly continuous foliation with $C^r$ leaves on a manifold $M_1$, $W_2$ is a uniformly continuous foliations with $C^r$ leaves on a manifold $M_2$ and $h\colon M_1\to M_2$ is homeomorphism which sends $W_1$ to $W_2$ then we will write 
$$h\in\Diff^r_{W_1,W_2}(M_1, M_2)$$
 if the restrictions of $h$ to the leaves of $W_1$ and their inverses are uniformly $C^r$.

The following is a straightforward corollary of the Journ\'e's Lemma, which is widely used in smooth dynamics.
\begin{cor}
\label{cor_journe}
Let $W_1$ and $V_1$ be mutually transverse uniformly continuous foliations with $C^r$ leaves on a manifold $M_1$ and $W_2$ and $V_2$ be mutually transverse uniformly continuous foliations with $C^r$ leaves on a manifold $M_2$.  Assume that  a homeomorphism $h\colon M_1\to M_2$ belongs to both $\Diff^r_{W_1,W_2}(M_1, M_2)$ and $\Diff^r_{V_1,V_2}(M_1, M_2)$. Then $h$ is $C^{r-\e}$ diffeomorphism for any $\e>0$.
\end{cor}

\subsection{Affine structures on expanding foliations}
\label{sec_affine_structures}

Let $W$ be a one dimensional uniformly continuous foliation with $C^r$ leaves on a complete (not necessarily compact) manifold $M$. Also let $f\colon M\to M$ be a diffeomorphism which leaves $W$ invariant  and uniformly expands the leaves of $W$. Then we say that $W$ is an {\it expanding foliation for} $f$. An expanding foliation $W$ can be equipped with dynamical densities which are defined using telescopic products of Jacobians of $f|_W$ as follows
$$
\rho_x(y)=\prod_{n\ge 1}\frac{D_Wf(f^{-n}(x))}{D_Wf(f^{-n}(y))},\, \,\,x\in M,\;\; y\in W(x),
$$
where $D_Wf$ is the Jacobian of the restriction $DF|_{TW}$.
\begin{lemma}[\cite{L1}, Lemma 4.3]
\label{lemma_density}
If $f$ is a uniformly $C^r$ diffeomorphism and $W$ is an expanding foliation for $f$ then the dynamical densities $\rho_x(\cdot)$, $x\in M$, are uniformly $C^{r-1}$ on $W(x,R)$ for any $R>0$.
\end{lemma}
It is easy to check that these families of densities are unique in the class of densities which posses the following properties
\begin{enumerate}
\item $\rho_x(x)=1$, $x\in M$;
\item $\rho_x(\cdot)$, $x\in M$, are uniformly continuous on $W(x, R)$ for a fixed $R>0$;
\item $\rho_{f(x)}(f(y))=\frac{D_Wf(x)}{D_Wf(y)}\rho_x(y)$ for all $x\in M$ and $y\in W(x)$.
\end{enumerate} 
We refer to~\cite{L1} for more information on dynamical densities and relation to SRB measures.

\begin{lemma}
\label{lemma_1d_bootstrap}
Let $W_1$ and $W_2$ be one dimensional expanding foliations for $g\colon M_1\to M_1$ and $f\colon M_2\to M_2$, respectively. Assume that $g$ is conjugate to $f$ via a homeomorphism $h\in \Diff^1_{W_1,W_2}(M_1, M_2)$. Also assume that $f$ and $g$ are uniformly $C^r$ diffeomorphisms. Then $h\in \Diff^{r}_{W_1,W_2}(M_1, M_2)$.
\end{lemma}
\begin{proof}
Let $\rho_x(\cdot)$ be the family of dynamical densities for $g$. Denote by $D_Wh$ the Jacobian of the restriction $Dh|_{TW_1}\colon TW_1\to TW_2$. Then
$$
\tilde\rho_{h(x)}(h(y))\stackrel{\mathrm{def}}{=}\frac{D_Wh(x)}{D_Wh(y)}\rho_x(y)
$$
satisfies all properties of dynamical densities for $f$ and, hence, is the unique family of dynamical densities for $f$. By Lemma~\ref{lemma_density}, both $\rho_x(\cdot)$ and $\tilde \rho_{h(x)}(\cdot)$ are uniformly $C^{r-1}$ along $W_1$ and $W_2$, respectively. Hence $D_Wh$ is uniformly $C^{r-1}$. Because we can also apply the same argument to $h^{-1}$ we conclude that $h\in \Diff^{r}_{W_1,W_2}(M_1, M_2)$.
\end{proof}

\subsection{Survival of the fine splitting under $C^1$ small perturbations}
\label{sec_survival_3d}
We assume that hyperbolic automorphism $L\colon\T^3\to\T^3$ has real eigenvalues $\lambda_1$, $\lambda_2$ and $\lambda_3$ such that
\begin{gather}\label{eq_lambdas}
0<\lambda_1<1<\lambda_2<\lambda_3
\end{gather}
In this case the unstable distribution splits as follows
$$
E^u_L=E^{wu}_L\oplus E^{uu}_L,
$$
where $E^{wu}_L$ corresponds to the eigendirection of $\lambda_2$ and $E^{uu}_L$ corresponds to the eigendirection of $\lambda_3$. Let $f\colon\T^3\to\T^3$ be a perturbation of $L$. It well known (see \eg~\cite[Chapter 3]{Pbook}) that if $f$ is sufficiently $C^1$-close to $L$ then $f$ is Anosov and this splitting survives, \ie
$$
E^u_f=E^{wu}_f\oplus E^{uu}_f,
$$
where the expansion along $E^{wu}_f$ is close to $\lambda_2$ and the expansion along $E^{uu}_f$ is close to $\lambda_3$. Distributions $E^{wu}_f$ anf $E^{uu}_f$ integrate uniquely to {\it weak unstable} and {\it strong unstable} foliations $W^{wu}_f$ and $W^{uu}_f$, respectively. While integrability of $E^{uu}_f$ follows from general theory (see \eg~\cite[Chapter 4]{Pbook}) integrability of $E^{wu}_f$ is more subtle (see \eg~\cite[Lemma 1]{GG}, see~\cite[Corollary 2.2]{LW} for a proof of a more general result). Both of foliations $W^{wu}_f$ and $W^{uu}_f$ subfoliate the 2-dimensional unstable foliation $W^u_f$. Foliation $W^{uu}_f$ is a continuous foliation with smooth leaves. However $W^{wu}_f$ is a continuous foliation with only $C^{1+\textup{H\"older}}$ leaves (we refer to~\cite{JPL} for a thorough discussion of the lack of regularity phenomenon).

\subsection{Cohomological equation over Diophanitine translations on the torus}
\label{sec_herman}

Consider a translation $T\colon\T^m\to\T^m$ on the $m$-torus given by
$$
x\mapsto x+\vec\beta \mod \Z^m,
$$
where $\vec\beta$
 is a vector which satisfies the following Diophantine condition
\begin{equation}
\label{eq_diophantine}
\left|\langle\vec\beta,\vec p\rangle-q\right|>\frac{c\,\,\,}{|\vec p\,|^m}
\end{equation}
for all $\vec p\in\Z^m\backslash\{0\}$ and $q\in\Z$, and some $c=c(\beta)>0$.

Consider a function $a\colon\T^m\to \R$ with zero average. The equation 
\begin{equation}
\label{eq_cohomological}
\varphi(Tx)-\varphi(x)=a(x)
\end{equation}
is called {\it cohomological equation}. If $a\in C^r(\T^m)$  then this equation admits a solution $\varphi\in C^{r-m-\e}(\T^m)$ for any $\e>0$~\cite{R} (also see~\cite[Proposition A.8.1]{Her}). Moreover, this solution is unique in $C^0(\T^m)$ up to an additive constant.

\section{Proof of the Main Theorem}
\label{sec_3}
Note that by passing to an appropriate (possibly negative) iterate of $L$ we may assume that the spectrum of $L$ satisfies~(\ref{eq_lambdas}). 

\subsection{Gogolev-Guysinsky result}
\label{sec_gg}
Our starting point is a weaker form of local rigidity established in~\cite{GG}. It was shown that for sufficiently $C^1$ small perturbations $f$ with vanishing periodic data obstructions, the conjugacy $h$ between $L$ and $f$ is $C^{1+\e}$, where $\e$ depends on $L$ only. In particular,
$$
h(W_L^{*})=W_f^{*},\,\mbox{for}\,\, \,*=s, wu, uu.
$$
Note that, by Lemma~\ref{lemma_1d_bootstrap}, we have 
\begin{equation}
\label{eq_s_uu_smooth}
h\in\Diff^\infty_{W_L^s, W_f^s}(\T^3)\cap\Diff^\infty_{W_L^{uu}, W_f^{uu}}(\T^3).
\end{equation}
Therefore our goal is to show that $W^{wu}_f$ is a continuous foliation with smooth leaves. After that, again by Lemma~\ref{lemma_1d_bootstrap}, we would have $h\in\Diff^\infty_{W_L^{wu}, W_f^{wu}}(\T^3)$ and Corollary~\ref{cor_journe} would imply that $h$ is a $C^\infty$ diffeomorphism.

\subsection{Smoothness of strong unstable foliation}
\label{sec_strong_unstable}

\begin{lemma}
 \label{lemma_strong_unstable}
 The strong unstable foliation $W_f^{uu}$ is a smooth foliation.
\end{lemma}
\begin{proof}
In our proof we will use the following well-known fact: a foliation is smooth, that is, given by smooth charts, if and only if it has smooth leaves and its holonomy homeomorphisms are smooth.
\\
{\it Step 1.} Recall that $W^{uu}_f$ subfoliates $W^u_f$. By~\cite[Proposition 3.9]{KS2}, for each $x\in M$ the restriction of $W^{uu}_f$ to $W^u_f(x)$ is a uniformly smooth foliation.
\\
{\it Step 2.} 
Let $W^{s+uu}_L$ be the integral foliation for $E^s_L\oplus E^{uu}_L$. Define
$$
W^{s+uu}_f=h(W^{s+uu}_L).
$$
Clearly $W^{s+uu}_f$ is a continuous foliation with $C^{1+\e}$ leaves. Note that $W^s_f$ and $W^{uu}_f$ are continuous foliations with smooth leaves which subfoliate $W^{s+uu}_f$. Hence, by representing local leaves of $W^{s+uu}_f$ as graphs and applying Journ\'e's Lemma we conclude that $W^{s+uu}_f$ has, in fact, smooth leaves.
Now, by~(\ref{eq_s_uu_smooth}) and Corollary~\ref{cor_journe}, we obtain that the restrictions
$$
h|_{W_L^{s+uu}(x)}\colon W_L^{s+uu}(x)\to W_f^{s+uu}(h(x))
$$
are uniformly $C^\infty$ diffeomorphisms. Hence for each $x\in M$ the restriction of $W^{uu}_f$ to $W^{s+uu}_f(x)$ is a uniformly smooth foliation.
\\
{\it Step 3.} Let $T_1$ and $T_2$ be a pair of 2-dimensional transversals to $W_f^{uu}$ and let
$$
\pi^{uu}\colon T_1\to T_2
$$
be a holonomy along strong unstable foliation. Let
$$
W_i=T_i\cap W^u,\, V_i=T_i\cap W^{s+uu},\,\, i=1,2.
$$
Then by previous steps we have
$$
\pi^{uu}\in \Diff^\infty_{W_1,W_2}(T_1, T_2)\cap\Diff^\infty_{V_1,V_2}(T_1, T_2).
$$
By Corollary~\ref{cor_journe}, $\pi^{uu}$ is a $C^\infty$ diffeomorphism and, hence, $W_f^{uu}$ is a $C^\infty$ foliation on $\T^3$. 
\end{proof}

\subsection{Anosov factor dynamics and bootstrap of quotient conjugacy}
\label{sec_factor_dynamics}
The goal of this subsection is to introduce the quotient conjugacy $\bar h$ and to prove Lemma~\ref{lemma_hbar} below.
\subsubsection{Quotient conjugacy}
\label{sec_quotient_conjugacy}
The embedding $\R^2\subset\R^3$ given by $(x,y)\mapsto (x,y,z)$ induces an embedding $\T^2\subset\T^3$. Foliation $W^{uu}_L$ is transverse to $\T^2$ because it is a totally irrational foliation. Recall that $E^{uu}_f$ depends continuously on $f$. It follows that if $f$ is sufficiently $C^1$ close to $L$ the distribution $E^{uu}_f$ is also transverse to $\T^2$. Hence, provided that $f$ is sufficiently $C^1$ close to $L$, we  have that $W^{uu}_f$ is transverse to $\T^2$.

We define the quotient conjugacy $\bar h\colon\T^2\to\T^2$ by taking the composite map
$$
\T^2\stackrel{h}{\longrightarrow} h(\T^2)\stackrel{\pi^{uu}}{\longrightarrow}\T^2.
$$
Here $\pi^{uu}$ is the ``short" holonomy along $W^{uu}_f$. This holonomy is well defined because $h$ is close to $id_{\T^3}$ and, hence, the local leaves $W^{uu, loc}_f(x)$, $x\in h(\T^2)$, intersect $\T^2$ exactly once. This definition is schematically illustrated on Figure~\ref{fig1}. 

\begin{figure}[hb]
  \centering
  \includegraphics{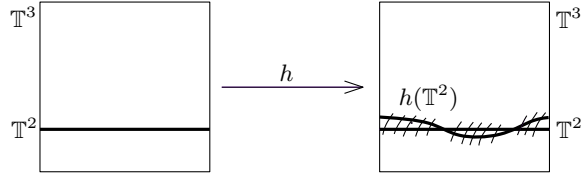}
  \caption
   {Definition of $\bar h$.}
   \label{fig1}
\end{figure}

By Lemma~\ref{lemma_strong_unstable}, $\bar h$ is as smooth as $h$, that is, $C^{1+\e}$. In fact, we have the following result.

\begin{lemma}
\label{lemma_hbar}
Quotient conjugacy $\bar h\colon \T^2\to\T^2$ is a $C^\infty$ diffeomorphism.
\end{lemma}
Note that $\bar h$ is not a conjugacy of Anosov diffeomorphisms. The basic idea of the proof of Lemma~\ref{lemma_hbar} is to lift $\bar h$ to a diffeomorphism which is a conjugacy between two Anosov diffeomorpisms of non-compact surfaces and then proceed with a fairly standard argument for bootstrap of regularity in the non-compact setting.
To prove Lemma~\ref{lemma_hbar} we need to explain the diagram depicted on Figure~\ref{fig2}.

\begin{figure}[hb]
  \centering
  \includegraphics{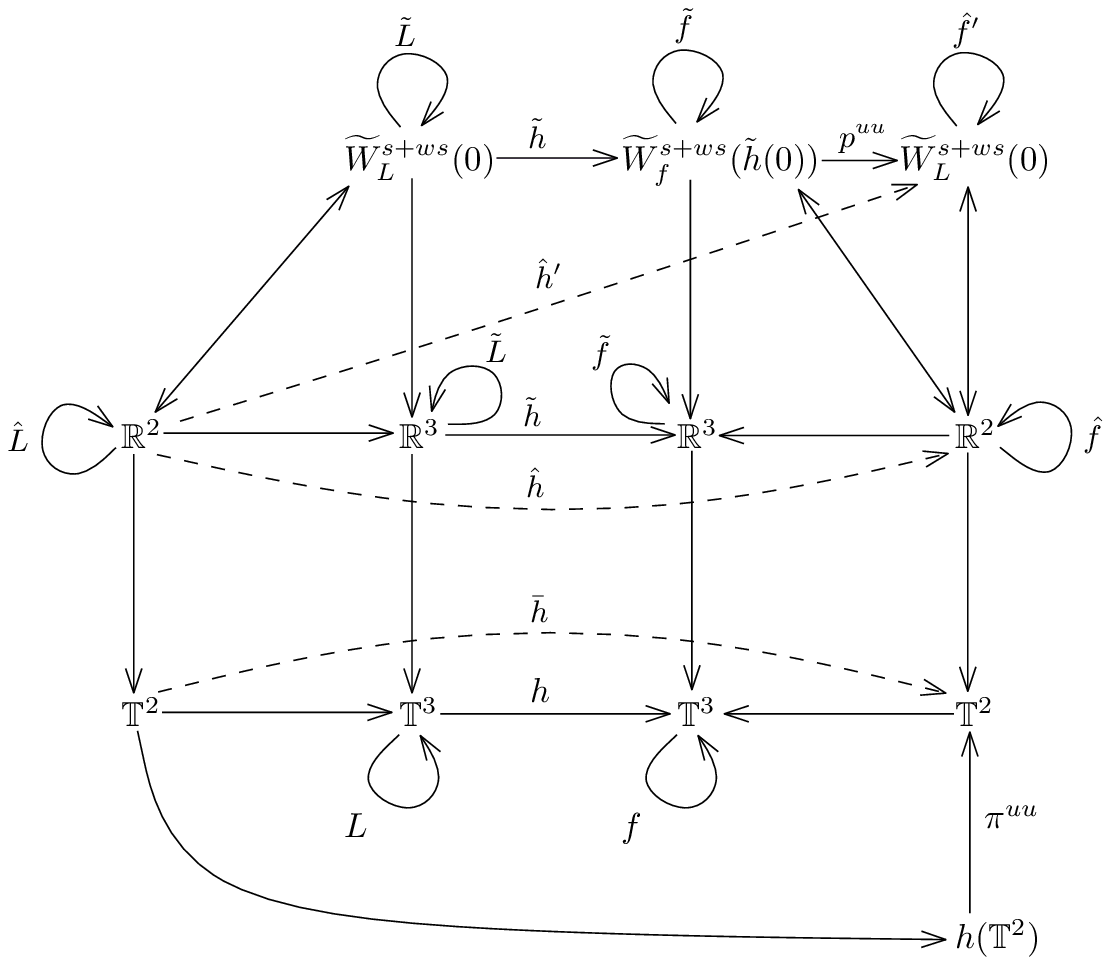}
  \caption
   {Diagram.}
   \label{fig2}
\end{figure}

\subsubsection{Lifts}
We pick lifts $\tilde L$, $\tilde f$ and $\tilde h$ of $L$, $f$ and $h$ to the universal cover $\R^3$ so that $\tilde L(0)=0$, $\tilde h$ is $C^0$ close to $id_{\R^3}$ and $\tilde h\circ \tilde L=\tilde f\circ\tilde h$.  Let $W^{s+wu}_L$ be the integral foliation for $E^s_L\oplus E^{wu}_L$ and let $W^{s+wu}_f=h(W^{s+wu}_L)$. We write $\widetilde W^{s+wu}_L$ and $\widetilde W^{s+wu}_f$ for the lifts of these foliations to $\R^3$. Note that by construction $\widetilde W^{s+wu}_L(0)$ is invariant under $\tilde L$ and $\widetilde W^{s+wu}_f(\tilde h(0))$ is invariant under $\tilde f$.

\subsubsection{Dynamics on the spaces of strong unstable leaves}
Note that given a 2-dimensional submanifold $\TT\subset\R^3$ which is a global transversal to $\widetilde W^{uu}_f$ we can identify $\TT$ with the space of strong unstable leaves via $\widetilde W^{uu}_f(x)\mapsto \widetilde W^{uu}_f(x)\cap \TT$. Further $\tilde f\colon \R^3\to\R^3$ induces a diffeomorphism of $\TT$. Note that the induced map is the composite of $\tilde f|_\TT$ and strong unstable holonomy
\begin{equation}
\label{eq_decomposition}
 \TT\stackrel{\tilde f}{\longrightarrow} \tilde f(\TT)\stackrel{uu-\mbox{holonomy}}{\xrightarrow{\hspace*{3cm}}}\TT.
\end{equation}
We apply this construction to the following three transversals --- $\R^2\subset \R^3$, $\widetilde W^{s+wu}_L(0)$ and $\widetilde W^{s+wu}_f(\tilde h(0))$ --- and obtain the induced diffeomorphisms as explained above:
$$
\hat f\colon\R^2\to\R^2,\;\; \hat f'\colon \widetilde W^{s+wu}_L(0)\to \widetilde W^{s+wu}_L(0),\;\; \tilde f\colon
\widetilde W^{s+wu}_f(\tilde h(0))\to \widetilde W^{s+wu}_f(\tilde h(0))
$$
Note the the last one is merely the restriction of $\tilde f\colon\R^3\to\R^3$ to $\widetilde W^{s+wu}_f(\tilde h(0))$ and, hence, we also denote it by $\tilde f$. These diffeomorphisms appear in the top-right corner of the diagram on Figure~\ref{fig2}. By construction~(\ref{eq_decomposition}) they are all conjugate via strong unstable holonomies, which form a commutative triangle in the top-right corner of the diagram on Figure~\ref{fig2}.

Fully analogous considerations apply to $\tilde L$, strong unstable foliation $\widetilde W^{uu}_L$ and transversals $\R^2$ and $\widetilde W^{s+wu}_L(0)$ which yield the diffeomorphisms
$\hat L\colon\R^2\to\R^2$ and $\tilde L\colon \widetilde W^{s+wu}_L(0)\to \widetilde W^{s+wu}_L(0)$. These two diffeomorphisms are conjugate via strong unstable holonomy along $\widetilde W^{uu}_L$ as indicated in the top-left corner of the diagram on Figure~\ref{fig2}.

\subsubsection{Lifts of the quotient conjugacy $\bar h$}

\label{sec_barh}
The lift $\tilde h\colon \R^3\to\R^3$ sends strong unstable leaves of $\tilde L$ to the strong unstable leaves of $\tilde f$ and, hence, induces a conjugacy $\hat h\colon\R^2\to\R^2$; \ie
$$
\hat h\circ\hat L=\hat f\circ\hat h.
$$
It is immediate from our definitions that $\hat h$ is a lift of $\bar h$ with respect to the covering map $\R^2\to\T^2$.
By composing $\hat h$ with holonomy along $\widetilde W^{uu}_f$ we obtain 
$\hat h'\colon\R^2\to\widetilde W^{s+wu}_L(0)$ which conjugates $\hat L$ and $\hat f'$. 

We have now fully explained the diagram on Figure~\ref{fig2}. Recall that by Lemma~\ref{lemma_strong_unstable} that holonomy along $\widetilde W^{uu}_f$  is smooth. Hence smoothness of $\hat h'$ implies smoothness of $\hat h$, which, in its turn, implies smoothness of $\bar h$.

\subsubsection{A uniformly smooth Anosov diffeomorphism of a non-compact surface} 
Let $p^{uu}\colon\widetilde W^{s+wu}_f(\tilde h(0))\to \widetilde W^{s+wu}_L(0)$ be the holonomy along $\widetilde W^{uu}_f$. Recall that
\begin{equation}
\label{eq_fhatprime}
\hat f'=p^{uu}\circ\tilde f|_{\widetilde W^{s+wu}_f(\tilde h(0))}\circ (p^{uu})^{-1},
\end{equation}
We equip both $\widetilde W^{s+wu}_L(0)$ and $\widetilde W^{s+wu}_f(\tilde h(0))$ with induced Riemannian metric.

\begin{lemma}
 \label{lemma_hat_f_prime}
 Diffeomorphism $\hat f'$ is a smooth Anosov diffeomorphism with uniformly bounded derivatives of all orders.
\end{lemma}
\begin{remark}
Note that diffeomorphism $\tilde f|_{\widetilde W^{s+wu}_f(\tilde h(0))}$ is clearly Anosov, however the surface $\widetilde W^{s+wu}_f(\tilde h(0))$ is merely $C^{1+\e}$, hence, it does not make sense to talk about higher regularity of $\tilde f|_{\widetilde W^{s+wu}_f(\tilde h(0))}$.
\end{remark}
\begin{proof}
The holonomy $p^{uu}$ is $C^{1+\e}$ diffeomorphism by Lemma~\ref{lemma_strong_unstable}. In fact, $p^{uu}$ is uniformly $C^1$, \ie first derivatives of $p^{uu}$ are uniformly bounded. 

 To check uniformity recall that $\widetilde W^{s+wu}_f(\tilde h(0))=\tilde h(\widetilde W^{s+wu}_L(0))$. It follows that $\widetilde W^{s+wu}_f(\tilde h(0))$ is bounded distance away from $\widetilde W^{s+wu}_L(0)$. And if we let $d^{uu}$ be the induced distance on the leaves of $\widetilde W^{uu}_f$ then 
 $$
 c=\sup_{x\in \widetilde W^{s+wu}_f(\tilde h(0))}d^{uu}(x,p^{uu}(x))
 $$
 is finite. We can view the holonomy $p^{uu}$ as being glued out of local holonomies 
$ \widetilde W^{s+wu,loc}_f(x)\to \widetilde W^{s+wu,loc}_L(y)$, with $d^{uu}(x,y)\le R$, where $x\in \widetilde W^{s+wu}_f(\tilde h(0))$ and 
$y=p^{uu}(x)=\widetilde W^{uu}_f(x)\cap \widetilde W^{s+wu}_L(0)$.
 Because these holonomies are between local leaves we can drop all tildes and conclude that each local holonomy of $p^{uu}$ belongs to a (strictly larger) family of holonomies
$$
\mathcal P=\{\,p^{uu}_{x,y}\colon W^{s+wu, loc}_f(x)\to W^{s+wu, loc}_L(y); \,\,\,d^{uu}(x,y)\le c\}
$$
Each local holonomy from $\mathcal P$ is uniformly smooth. Further, because foliations $W^{s+wu}_f$ and $W^{s+wu}_L$ are uniformly $C^{1+\e}$ and $W^{uu}_f$ is uniformly smooth by Lemma~\ref{lemma_strong_unstable}, we conclude that $p^{uu}_{x,y}$ vary continuously in $C^1$ topology with respect to $x\in\T^3$ and $y\in W^{uu}(x, c)$. Hence, by compactness, the holonomies from $\mathcal P$ are uniformly uniformly $C^1$. We conclude that $p^{uu}$ is uniformly $C^1$.
 
 Obviously, $\tilde f|_{\widetilde W^{s+wu}_f(\tilde h(0))}$ is Anosov. Therefore, by (\ref{eq_fhatprime}), $\hat f'$ is uniformly $C^1$ conjugate to an Anosov diffeomorphism and, hence, is Anosov.
 (To be more precise, we claim that $\hat f'$ is Anosov with respect to the induced Riemannian metric on $\widetilde W^{s+wu}_L(0)$. This is why $C^1$ uniformity of the conjugacy $p^{uu}$ is important here.)
 
 It remains to see that $\hat f'$ is uniformly smooth. For this we use the description of $\hat f'$ as the induced map on the space of strong unstable leaves. Namely recall that by~(\ref{eq_decomposition}) $\hat f'$ is the composition
 $$
 \hat f'\colon \widetilde W^{s+wu}_L(0)\stackrel{\tilde f}{\longrightarrow} \tilde f(\widetilde W^{s+wu}_L(0))\stackrel{q^{uu}}{\longrightarrow}\widetilde W^{s+wu}_L(0),
 $$
 where $q^{uu}$ is the holonomy along $\widetilde W^{uu}_f$. The first diffeomorphism of the composition is uniformly smooth. Because $\tilde f$ is uniformly close to $\tilde L$, the distance between $x\in \tilde f(\widetilde W^{s+wu}_L(0))$ and $q^{uu}(x)\in\widetilde W_L^{s+wu}(0)$ along $\widetilde W^{uu}_f(x)$ is uniformly bounded. 
 
Therefore we can apply the same argument, which we used to show that $p^{uu}$ is uniformly $C^1$, to the holonomy $q^{uu}$. Indeed, foliations $ \tilde f(\widetilde W^{s+wu}_L)$ and $\widetilde W_L^{s+wu}$ are uniformly $C^\infty$. Hence this argument yields uniform smoothness of $q^{uu}$. We conclude that $\hat f'$ is indeed uniformly smooth Anosov diffeomorphism.
\end{proof}
\begin{remark}
 The argument used to prove Lemma~\ref{lemma_hat_f_prime} does not work to establish the Anosov property and uniform smoothness of $\hat f$ because corresponding strong unstable holonomies between $\widetilde W^{s+wu}_f(\tilde h(0))$ and $\R^2$ are unbounded and thus, a priori, may have unbounded derivatives. This is the reason behind introducing and working with $\hat f'$.
\end{remark}

\subsubsection{Proof of Lemma~\ref{lemma_hbar}}
We have already explained in Subsection~\ref{sec_barh} that in order to prove Lemma~\ref{lemma_hbar} we only need to establish smoothness of $\hat h'$.

Recall that $\hat h'\circ\hat L=\hat f'\circ\hat h'$. By definition $\hat h'$ is $C^{1+\e}$.  Lemma~\ref{lemma_hat_f_prime} verifies the assumption of Lemma~\ref{lemma_1d_bootstrap}, which applies and yields smoothness of $\hat h'$ along the one dimensional expanding and contracting foliations. Therefore, by Corollary~\ref{cor_journe}, $\hat h'$ is a smooth diffeomorphism. \hfill$\square$

\subsection{Setting up the cohomological equation}
\label{sec_cohomological_equation}
\subsubsection{A different role of the quotient conjugacy $\bar h$: conjugacy of return maps $R$ and $T$}
Let $\bar h\colon\T^2\to\T^2$ be the quotient conjugacy defined in~\ref{sec_quotient_conjugacy}. We orient $W^{uu}_L$ and $W^{uu}_f$ so that $h$ preserves the orientation. Consider the flows along $W^{uu}_L$ and $W^{uu}_f$. By discussion in~\ref{sec_quotient_conjugacy}, $\T^2\subset\T^3$ is a transversal for these flows and, hence, we can consider first return maps $T\colon \T^2\to\T^2$ and $R\colon \T^2\to\T^2$, respectively. By Lemma~\ref{lemma_strong_unstable}, $R$ is a smooth diffeomorphism. Obviously, $T$ is a translation on $\T^2$. It is also evident from Figure~\ref{fig3} that $\bar h$ conjugates the return maps
\begin{equation}
\label{eq_hbar}
\bar h\circ T= R\circ\bar h
\end{equation}

\begin{figure}[hb]
  \centering
  \includegraphics{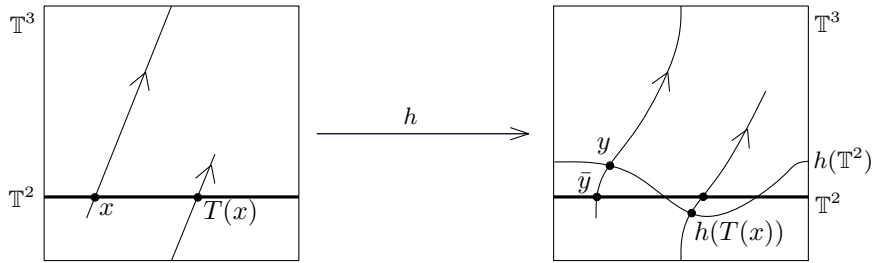}
  \caption
   {Diffeomorphism $\bar h$ conjugates the return maps. Here $y=h(x)$, $\bar y=\bar h(x)$ and the unlabeled point on the right is $\bar h(T(x))=R(\bar y)$.}
   \label{fig3}
\end{figure}

\subsubsection{The cohomological equation over $R$}
\label{sec_coh_over_R}
Let $g$ be the standard flat Riemannian metric on $\T^3$. We equip bundle $E^{uu}_f$ with the pull-back metric
\begin{equation}
 \label{eq_pullback_metric}
g^{uu}_f=(h^{-1})^*g|_{E^{uu}_L}
\end{equation}
and denote by $d^{uu}_f$ the induced metric on the leaves of $W^{uu}_f$. We also equip the transversal $\T^2\subset\T^3$ with the pull-back metric $(\bar h^{-1})^*g|_{\T^2}$. Note that the latter metric is smooth by Lemma~\ref{lemma_hbar}.

Let $W=\T^2\cap W^u_f$. We also pick an orientation for $W$. Note that by our choice of metric on $\T^2$, diffeomorphism $R\colon\T^2\to \T^2$ is an isometry. Therefore,
\begin{equation}
\label{eq_Rderivative}
\frac{\partial R}{\partial W}=1
\end{equation}

The leaves of $W^{wu}_f$ can be viewed as graphs over the leaves of $W$. Namely, for each $x_0\in\T^2$ pick a point $y_0\in W^{uu,loc}_f(x_0)$ and consider the holonomy
$$
\alpha_{x_0y_0} \colon W^{loc}(x_0)\to W^{wu, loc}_f(y_0),\, \alpha_{x_0y_0}(x_0)=y_0,
$$
along $W^{uu}_f$. Define
$$
\Phi_{x_0y_0}(x)=d_f^{uu}(x,\alpha_{x_0y_0}(x)),\, x\in W^{loc}(x_0).
$$
Then $W^{wu,loc}_f(y_0)$ is the graph of $\Phi_{x_0y_0}$ over $W^{loc}(x_0)$.

By our choice of metric on $E^{uu}_f$ the leaves of $W^{wu}_f$ are $d^{uu}_f$-equidistant (\ie $d^{uu}_f(x,y)$ only depends on the leaves $W^{wu}_f(x)$ and $W^{wu}_f(y)$ rather than particular points) within the leaves of $W^u_f$. It follows that 
\begin{equation}
\label{eq_small_phi}
\varphi(x)\stackrel{\textup{def}}{=}\frac{\Phi_{x_0y_0}(x)}{\partial W}, \,\,x\in W^{loc}(x_0)
\end{equation}
does not depend on the choice of $x_0$ and $y_0$. It is easy to see that $\varphi\colon\T^2\to\T^2$ is H\"older continuous. Also, as shown on Figure~\ref{fig4}, the fact that the weak unstable leaves are $d^{uu}_f$-equidistant implies that
\begin{equation}
\label{eq_Phi}
\Phi_{R(x_0)y_0'}(R(x))=\Phi_{x_0y_0}(x)-A(x)+const(y_0,y_0'),
\end{equation}
where
\begin{equation}
\label{eq_A}
A(x)\stackrel{\textup{def}}{=}d^{uu}_f(x,R(x)).
\end{equation}

\begin{figure}[hb]
  \centering
  \includegraphics{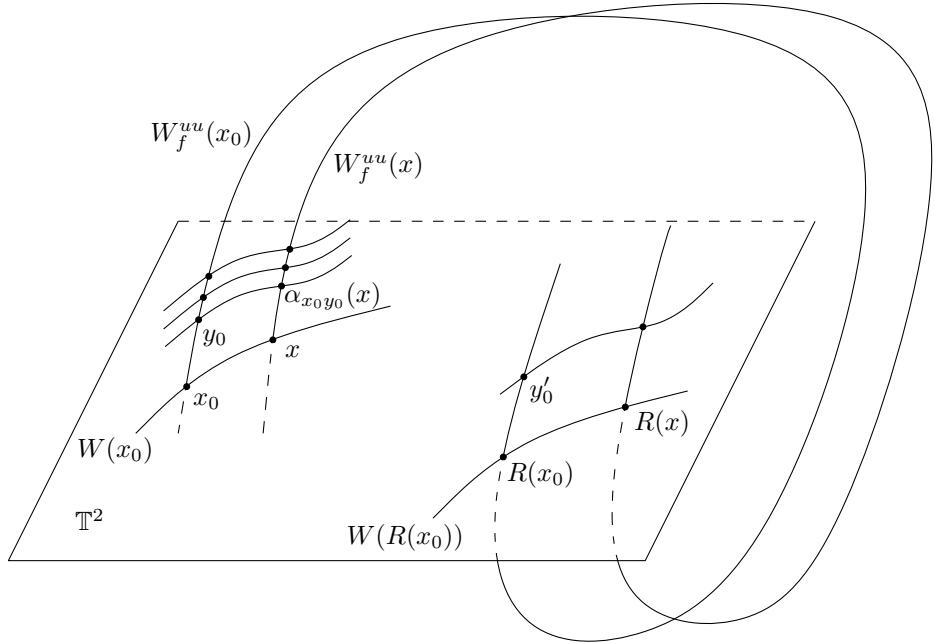}
  \caption
   {Proof of~(\ref{eq_Phi}). Equidistant weak unstable manifolds through $W^{uu}_f(x_0)$ are indicated.}
   \label{fig4}
\end{figure}

Differentiating~(\ref{eq_Phi}) along $W$ and using~(\ref{eq_Rderivative}) yields the cohomological equation
\begin{equation}
\label{eq_cohom_x}
\varphi(x)-\varphi(R(x))=a(x),
\end{equation}
where
\begin{equation}
\label{eq_a_small}
a(x)\stackrel{\textup{def}}{=}\frac{\partial A(x)}{\partial W}
\end{equation}

\subsection{Proof of the Main Theorem}
We will show that function $\varphi$ defined by~(\ref{eq_small_phi}) is smooth. After that we will deduce that $W^{wu}_f$ is a continuous foliation with smooth leaves. In fact, we will show that $W^{wu}_f$ is a smooth foliation. As was already explained in~\ref{sec_gg} this would complete the proof of the Main Theorem.

\begin{lemma}
 \label{lemma_pullback_smooth}
 The pull-back metric on the strong unstable distribution given by~(\ref{eq_pullback_metric}) is a smooth metric.
\end{lemma}
We prove this lemma at the end of the current section.
\begin{cor}
 \label{cor_a_smooth}
 Functions $A\colon\T^2\to\R$ and $a\colon\T^2\to\R$ given by~(\ref{eq_A}) and~(\ref{eq_a_small}), respectively, are smooth.
\end{cor}
\begin{proof}
 Indeed, smoothness of $A$ is immediate from smoothness of strong unstable foliation and the above lemma. The derivative $a$ is also smooth because $W$ is a smooth foliation (which follows from Lemma~\ref{lemma_hbar}).
\end{proof}

We rewrite cohomological equation~(\ref{eq_cohomological}) using~(\ref{eq_hbar}):
$$
\varphi(\bar h(y))-\varphi(\bar h(T(y)))=a(\bar h(y)).
$$
By letting $\bar\varphi=\varphi\circ\bar h$ and $\bar a=a\circ\bar h$ we obtain
\begin{equation}
 \label{eq_cohom_y}
\bar\varphi(y)-\bar\varphi(T(y))=\bar a(y).
\end{equation}
Function $\bar a$ is smooth by Lemma~\ref{lemma_hbar} and Corollary~\ref{cor_a_smooth}.

Now recall that $T$ is a translation
$$
y\mapsto y+(\beta_1,\beta_2) \mod \Z^2,
$$
where $(\beta_1,\beta_2,1)$ is the eigenvector of $L$ for the strong unstable eigenvalue $\lambda_3$. Because $\lambda_3$ is a root of an irreducible integral polynomial of degree 3, it follows from the generalized Liouville's Theorem (see, \eg~\cite[VI, Lemma~1A]{Sch}) that $(\beta_1,\beta_2)$ is a Diophantine vector in the sense of~(\ref{eq_diophantine}). Hence $\bar\varphi\colon\T^2\to\R$ is smooth by the regularity result discussed in Section~\ref{sec_herman}. Because $\bar h$ is smooth we conclude that $\varphi$ is also smooth.

Finally, recall that, by discussion in~\ref{sec_coh_over_R}, the leaves of $W^{wu}_f$ are graphs of the ``antiderivative of $\varphi$.'' Together with Lemma~\ref{lemma_pullback_smooth} this implies that $W^{wu}_f$ is a smooth foliation. Thus to finish the proof of the Main Theorem it remains to establish Lemma~\ref{lemma_pullback_smooth}.
\begin{proof}[Proof of Lemma~\ref{lemma_pullback_smooth}]
 We denote by $dy^2$  the restriction of the standard flat Riemannian metric $g$ to $E^{uu}_f$. By definition
$$
g^{uu}_f(y)=(D^{uu}h^{-1}(y))dy^2,
$$
where $D^{uu}h^{-1}(y)$ is the Jacobian of the restriction of the differential $Dh^{-1}$ to $E^{uu}_f(y)$. Recall that by Lemma~\ref{lemma_strong_unstable} $W^{uu}_f$ is smooth. Hence to prove the lemma it is sufficient to show that $D^{uu}h^{-1}$ is smooth.

Let $\xi=D^{uu}h^{-1}$. Differentiating the conjugacy equation
$$
L\circ h^{-1}=h^{-1}\circ f
$$
along strong unstable distribution yields
$$
D^{uu}L(h^{-1}(y)) \xi(y)=\xi(f(y)) D^{uu}f(y).
$$
By taking the logarithms and recalling that $\lambda_3$ is the strong unstable eigenvalue of $L$ we obtain
$$
\log\lambda_3+\log(\xi(y))=\log(\xi(fy))+\log D^{uu}f(y)
$$
or
$$
\log(\xi(fy))-\log(\xi(y))=\log D^{uu}f(y)-\log\lambda_3.
$$
We arrived at a cohomological equation over $f$ with smooth coboundary on the right. Hence, by regularity theory for cohomological equations over Anosov diffeomorphisms~\cite{LMM}, we obtain that $\log(\xi)$ and, hence, $\xi$ are smooth. 
\end{proof}

\subsection{Sketch of the proof of Addendum~\ref{add}}
As before we can assume that the spectrum of $L$ satisfies~(\ref{eq_lambdas}). Because the Galois group of the characteristic polynomial of $L$ acts transitively on the roots $\{\lambda_1,\lambda_2,\lambda_3\}$ the number $\log\lambda_3/\log\lambda_2$ is not an integer. Let
\begin{equation}
\label{eq_kappa}
\kappa=\left\lfloor\frac{\log\lambda_3}{\log\lambda_2}\right\rfloor.
\end{equation}
Gogolev-Guysinsky result described in~\ref{sec_gg}, in fact, yields $C^{r-\e}$ regularity of the conjugacy $h$ in the case when $r\le\kappa$. Thus it remains to consider the case $r\ge\kappa+3$. In this case Gogolev-Guysinsky result only yields $C^{\kappa+\delta}$ regularity of $h$, where $\delta>0$ is a small constant. (For this it is important that  (because $\kappa$ is not an integer) we have $\kappa<\log\lambda_3/\log\lambda_2$.) However, it is straightforward to verify that following the arguments of Section~\ref{sec_3} only ``$\e$-loss'' of regularity occurs up to the point when we need to solve cohomological equation~(\ref{eq_cohom_y}). Indeed, function $\bar a\colon \T^2\to\R$ is only $C^{r-1-\e}$ and, hence,~\ref{sec_herman} only yields $C^{r-3-\e}$ smoothness of $\bar \varphi$ (and hence $\varphi$). Consequently, $W^{wu}_f$ is a $C^{r-2-\e}$ foliation. Note that $r-2-\e\ge \kappa+1-\e$. Therefore the conjugacy $h$ is a $C^{\kappa+1-\e}$ diffeomorphism. This regularity exceeds de la Llave's critical threshold, \ie
$$
\kappa+1-\e>\frac{\log\lambda_3}{\log\lambda_2}.
$$
Hence Theorem~6.1 of~\cite{L1} applies and yields $C^{r-\e}$ regularity for $h$.

\section{Analytic local rigidity}

Here we explain how to obtain analytic local rigidity stated in Addendum~\ref{add_analytic}. First we need to recall some background on invariant local manifolds for the intermediate  distribution.

\subsection{Invariant families of local manifolds for the intermediate distribution}
Let $L\colon\T^3\to\T^3$ be an automorphism whose spectrum is given by~(\ref{eq_lambdas}) and let $f$ be a sufficiently $C^1$ small perturbation of $L$. Recall that by the discussion in~\ref{sec_survival_3d} there exists a $Df$-invariant splitting $E^s_f\oplus E^{wu}_f\oplus E^{uu}_f$. Next theorem summarizes slow local invariant manifold theory in our restricted setting. We refer the reader to~\cite{P, JPL, CFL} for general statements, discussion and proofs. Note that because $\log\lambda_3/\log\lambda_2$ is not an integer the non-resonance conditions needed for the general result hold automatically for sufficiently $C^1$ small perturbations $f$. 
\begin{theorem}[\cite{P}]
\label{thm_p}
Let $L$ and $f$ be as above and let $\kappa$ be the critical regularity given by~\textup{(\ref{eq_kappa})}. Then there exist a family of local manifolds $\{ V_f^{wu}(x), x\in\T^3\}$ such that
\begin{enumerate}
\item $V^{wu}_f(x)$, $x\in\T^3$, are uniformly smooth;
\item $T_xV^{wu}_f(x)=E^{wu}_f(x)$, $x\in\T^3$;
\item $f^{-1}(V^{wu}_f(x))\subset V^{wu}_f(f^{-1}(x))$.
\end{enumerate}
Furthermore, family $\{ V_f^{wu}(x), x\in\T^3\}$ is the only family of uniformly $C^{\kappa+1}$ local manifolds which satisfy 1 and 2 above.
\end{theorem}
\begin{add}[\cite{CFL}, Theorem 2.2]
\label{add_cfl}
Moreover, if $f$ is analytic then the local manifolds $V^{wu}_f(x)$, $x\in\T^3$, are uniformly analytic.
\label{add_slow_manifold}
\end{add}

\subsection{Proof of Addendum~\ref{add_analytic}}
Recall that we assume that $f\colon\T^3\to\T^3$ is a $C^1$ small analytic perturbation of the automorphism $L\colon\T^3\to\T^3$ for which the periodic data obstructions vanish.
\begin{enumerate}
\item[ Step 1.] By the proof of the Main Theorem $W^{wu}_f$ is a smooth foliation. Hence $\{ W_f^{wu, loc}(x), x\in\T^3\}$ is family of uniformly smooth local manifolds. By the uniqueness part of Theorem~\ref{thm_p}, we have
$$
W_f^{wu, loc}(x)=V^{wu}_f, x\in\T^3.
$$
\item[Step 2.] By Addendum~\ref{add_cfl}, $V^{wu}_f(x)$, $x\in\T^3$, are uniformly analytic. Hence $W^{wu}_f$ has analytic leaves.
\item[Step 3.] The foliaitons $W^s_f$ and $W^{uu}_f$ also have analytic leaves (see \eg discussion in~\cite{L4}). Hence, by the analytic version of Lemma~\ref{lemma_density}~(\cite[Lemma~4.3]{L1}) the dynamical densities on $W^s_f$, $W^{wu}_f$ and $W^{uu}_f$ are analytic.
\item[Step 4.] By following the arguments of Lemma~2.4 we obtain that the conjugacy is uniformly analytic along the leaves of $W^s_f$, $W^{wu}_f$ and $W^{uu}_f$.
\item[Step 5.] Finally we apply de la Llave's analytic version of Journ\'e's Lemma~\cite{L4}, first to weak unstable and strong unstable pair of foliations, and then to stable and unstable pair of foliations, and conclude that the conjugacy is analytic.
\end{enumerate}

\begin{remark}
Another argument for establishing analyticity along the full unstable foliation was pointed out to us by one of the anonymous referees. This is a bootstrap argument based on the proof of Theorem 6.1 of~\cite{L1}. One needs to observe that the series representation of the unstable derivative of the conjugacy $h$ converges on a small complex extension, and hence, yields analyticity along the unstable foliation.
\end{remark}

\end{document}